\documentclass[11pt]{article}

\usepackage{amssymb,amsmath,euscript,mathrsfs,dsfont}

\usepackage[numbers]{natbib}
\newtheorem{proposition}{Proposition}[section]
\newtheorem{lemma}[proposition]{Lemma}
\newtheorem{theorem}[proposition]{Theorem}

\newtheorem{corollary}[proposition]{Corollary}

\def\l{{\langle}}
\def\r{\rangle}

\def\R{{\mathbb R}}
\def\S{{\mathbb S}}

\def\E{{\mathbb E}}
\def\P{{\mathbb P}}

\makeatletter \@addtoreset{equation}{section} \makeatother
\newcommand {\qed}%
{%
    {}\hfill
    {}\hfill
    {$\square $}%
    \vspace {0.3cm}%
    \pagebreak [2]%
    \par
}%
\newenvironment{proof}[1]{%
    \vspace{0.3cm}%
    \pagebreak [2]%
    \par%
    \noindent {\bf  Proof~#1\ }}{\qed}%
\newenvironment{remark}{%
    \vspace{0.3cm} \pagebreak [2]%
    \par%
    \refstepcounter{proposition}
    \noindent%
    {\bf Remark~\theproposition\  }}{}%
\begin{document}

\title {Excursion Probabilities of Isotropic and Locally Isotropic Gaussian Random Fields on Manifolds}
\author{Dan Cheng\\ North Carolina State University }

\maketitle

\begin{abstract}
Let $X= \{X(p), p\in M\}$ be a centered Gaussian random
field, where $M$ is a smooth Riemannian manifold. For a suitable compact subset $D\subset M$,
we obtain the approximations to excursion probability $\P\{\sup_{p\in D} X(p) \ge u \}$, as $u\to \infty$,
for two cases: (i) $X$ is smooth and isotropic; (ii) $X$ is non-smooth and locally isotropic. For case (i), the expected
Euler characteristic approximation is formulated explicitly; while for case (ii),
it is shown that the asymptotics is similar to Pickands' approximation on Euclidean
space which involves Pickands' constant and the volume of $D$. These extend the results in \citep{Cheng:2014} from sphere to general Riemannian manifolds.
\end{abstract}



\section{Introduction}
For a Gaussian random field $X= \{X(p), p\in M\}$ living on some parameter space $M$, there is rich literature on studying the \emph{excursion probability} $\P\{\sup_{p\in D} X(p) \ge u \}$ as $u\to \infty$, where $D$ is a suitable compact subset of $M$  \citep{Adler00,AT07,AW09,Piterbarg96}. Traditional research usually focuses on Gaussian random fields living on Euclidean space. However, due to recent development in statistics and related fields such as geoscience, astronomy and neuroscience, Gaussian random fields parameterized on manifolds are becoming more and more useful for modelling many research objects. Motivated by both real applications and theoretical development \citep{MarinucciPecati11,Cheng:2014,Cohen12,Istas13,Malyarenko13,Yaglom61}, it is valuable to extend the study of excursion probability from Euclidean space to manifolds.

When $X$ is smooth and $M$ is a Riemannian manifold, \citet{TTA05} (see also Theorem 14.3.3 in \citep{AT07}) showed that the excursion probability can be approximated by the expected Euler characteristic of the excursion set of $X$ exceeding $u$. This verifies the so called ``Euler Characteristic Heuristic''. The expected Euler characteristic can be formulated by a nice but implicit form (see \citep{TA03} or Theorems 12.4.1 and 12.4.2 in \citep{AT07}), which involves Lipschitz-Killing curvatures of the excursion set. Since the Lipschitz-Killing curvatures depend on both the field $X$ and the parameter space $M$, it is usually hard to obtain their explicit expressions. In this paper, we consider $X$ to be isotropic on $M$ in the sense that its covariance function can be written as $C(p,q)=\rho(d_M^2(p,q))$ for any $p, q \in M$, where $\rho$ is a real function and $d_M$ is the geodesic distance on $M$. It is then showed that the Lipschitz-Killing curvatures of the excursion set can be simplified as depending only on $\rho'(0)$ and the geometry of $D$, see Theorem \ref{Thm:MEC M} below. This leads to a relatively simple formula for the expected Euler characteristic of the excursion set, which approximates the excursion probability with a super-exponentially small error as $u\to \infty$ (see Theorem \ref{Thm:Euler approximation}).

We also study another case when $X$ is non-smooth and $M$ is a Riemannian manifold. In particular, we assume the covariance function of $X$ satisfies the locally isotropic condition \eqref{Eq:Pickands condition M} below. Notice that since $X$ is non-smooth, the expected Euler characteristic approximation is no longer applicable. Instead, we shall employ arguments by charts to study the corresponding field under Euclidean local coordinates and then apply the result in \citep{ChanL06} to find an approximation to the excursion probability. As a related reference, we have found that \citet{MP97} studied the excursion probability over a manifold subset for Gaussian random fields parameterized on Euclidean space. Their theorems can be applied to obtain results similar to Theorem \ref{Thm:Pickands approximation M} in this paper if $X$ is the restriction on $M$ of some Gaussian field defined on Euclidean space of higher dimension. However, it is not clear if every locally isotropic Gaussian field on $M$ can be regarded as a restriction of some Gaussian field on Euclidean space satisfying the local property in (4) in \citep{MP97}. To see this, we refer to explanations in \citep{Cheng:2014} for the case when $M$ is an $N$-dimensional sphere (the case when $M$ is an arbitrary Riemannian manifold will make it even more complicated). On the other hand, in many real applications, it is natural to model a Gaussian field on $M$ directly and then explore its local properties on $M$ but not its extension on Euclidean space of higher dimension. In such sense, condition \eqref{Eq:Pickands condition M} is more natural and useful compared with (4) in \citep{MP97}.

The results in this paper extend those in \citep{Cheng:2014} from sphere to general Riemannian manifolds. The study for excursion probabilities of Gaussian fields on sphere in \citep{Cheng:2014} relies on specific spectral representation of the covariance \citep{Schoenberg42} and the spherical coordinates. To deal with Gaussian fields on general Riemannian manifolds, we explore directly the connection between the Riemannian metric and the (locally) isotropic property of the field to obtain the asymptotics of excursion probabilities (see Theorem \ref{Thm:Pickands approximation M} and Corollary \ref{Cor:Pickands approximation Dk} below). This new approach is more general and powerful.

\section{Smooth Isotropic Gaussian Random Fields on Manifolds}
\label{Sec:smooth}
\subsection{Preliminaries on Riemannian Manifolds}
\label{Sec:manifold}
We first introduce some notations and definitions on Riemannian manifolds. Let $(M,g)$ be an $N$-dimensional smooth Riemannian manifold, where $g$ is the Riemannian metric. Then for each $p\in M$, there is an inner product
\begin{equation*}
\begin{split}
g_p: \ T_pM\times T_pM &\rightarrow \R \\
 (\xi_p,  \sigma_p) &\mapsto g_p(\xi_p,  \sigma_p),
\end{split}
\end{equation*}
where $T_pM$ is the tangent space of $M$ at $p$. To simplify the notation, we denote by $\l \xi_p,  \sigma_p \r$ the inner product $g_p(\xi_p,  \sigma_p)$. In particular, for a coordinate system with bisis $\{\frac{\partial}{\partial x_i} |_p \}_{1\le i, j\le N}$, we let
\[
g_{ij}(p) := g_p\left(\frac{\partial}{\partial x_i} \Big|_p, \frac{\partial}{\partial x_j} \Big|_p\right) = \left\langle \frac{\partial}{\partial x_i} \Big|_p,  \frac{\partial}{\partial x_j} \Big|_p \right\rangle,
\]
which are called the components of $g$ at $p$ under this coordinate system. Denote the $N\times N$ matrix by  $G(p)=(g_{ij}(p))_{1\le i, j\le N}$, which is symmetric and positive definite due to the definition of $g$. For a real-valued smooth function $f$ on $M$, let $f_i = \frac{\partial}{\partial x_i}f$ and $f_{ij} = \frac{\partial^2}{\partial x_i\partial x_j}f$.

Notice that, the Riemannian metric $g$ itself is not a true metric on $M$. However, it does induce a metric on $M$ in the following way. Firstly, we can define the length of a differentiable curve $\gamma: [a, b] \rightarrow M$ by
\[
L(\gamma) = \int_a^b \sqrt{\l \gamma'(t), \gamma'(t)\r} dt.
\]
Then the distance on $M$ induced by the Riemannian metric $g$, denoted by $d_M(\cdot, \cdot)$, is defined by
\[
d_M(p,q) = \inf L(\gamma), \quad p, q\in M,
\]
where the infimum extends over all differentiable curves $\gamma$ beginning at $p \in M$ and ending at $q\in M$.

\subsection{Isotropic Gaussian Random Fields on Riemannian Manifolds}
Let $X=\{X(p), p\in M\}$ be a smooth centered Gaussian random field parameterized over an $N$-dimensional Riemannian manifold $(M,g)$. We define $X$ to be \emph{isotropic} over $M$ if the covariance function satisfies
\begin{equation}\label{Eq:isotropy}
C(p,q) = \rho\left(d_M^2(p,q)\right), \quad p,q\in M,
\end{equation}
where $\rho$ is a real function on $[0,\infty)$. Relation \eqref{Eq:isotropy} implies that $C(p,q)$ only depends on $d_M(p,q)$ and  hence behaves isotropically over $M$. Besides the well-known Euclidean case, we refer to \citep{MarinucciPecati11,Cheng:2014} and \citep{Cohen12} respectively for isotropic Gaussian fields on sphere and hyperbolic space satisfying \eqref{Eq:isotropy}. For general Riemannian manifolds, isotropic Gaussian fields and covariances can be constructed via spectral representations; see \citep{Yaglom61,Yaglom87,Yadrenko83,Malyarenko13}.

As discussed in \citep{AT07} (see Corollaries 11.3.3 and 11.3.5 therein), in order to apply the Kac-Rice metatheorem and the Morse theory to establish the expected Euler characteristic approximation, we need the following smoothness and regularity conditions for $X$.

\begin{itemize}
\item[({\bf H}1).] $X(\cdot) \in C^2(M)$ almost surely and there exist positive constants $K$ and $\eta$ such that
\begin{equation*}
\E(X_{ij}(p)-X_{ij}(q))^2 \leq K |\log (d_M(p,q))|^{-(1+\eta)}, \quad \forall p,q\in M,\ i, j= 1, \ldots, N.
\end{equation*}

\item[({\bf H}2).]  For every $p\in M$, the Gaussian vector $(X_i(p), X_{jk}(p), 1\le i\le N, 1\leq j\leq k\leq N)$ is  non-degenerate.
\end{itemize}
In addition, to make the compact subset $D\subset M$ suitable, we shall assume that $D$ is a regular stratified, locally convex submanifold. We refer to pages 198 and 189 in \citep{AT07} for rigorous definitions of ``regular stratified'' and ``locally convex''. Roughly speaking, we may let $D$ be piecewise smooth and convex.

Next, we introduce the Lipschitz-Killing curvatures (or ``intrinsic volumes'') of $D$. For $j=0, \ldots, {\rm dim}(D)$, the $j$-th Lipschitz-Killing curvature $\mathcal{L}_j (D)$ can be viewed as the measure of the $j$-dimensional boundary of $D$. Generally, if $D$ is a $k$-dimensional, regular stratified, locally convex manifold, then $\mathcal{L}_0 (D)$ is the Euler characteristic, $\mathcal{L}_k (D)$ is the volume and $\mathcal{L}_{k-1} (D)$ is half of the surface area. For the other $\mathcal{L}_j (T)$, $1\le j\le k-2$, we can apply Steiner's formula (see page 142 in \citep{AT07}) to find their values. We refer to \citep{AT07} for the general definition of $\mathcal{L}_j (D)$ and specific examples on evaluating them explicitly as well.

Let $\chi(A_u(X, D))$ be the Euler characteristic of excursion set $A_u(X,D) = \{x\in D: X(p)\geq u\}$ (cf. \citep{AT07}). Denote by $H_j(x)$ the Hermite polynomial of degree $j$, that is,
\[
H_j(x) = (-1)^j e ^{x^2/2} \frac{d^j}{dx^j}\big( e^{-x^2/2} \big).
\]
We obtain the following formula for $\chi(A_u(X, D))$ when $X$ is smooth and isotropic.

\begin{theorem}\label{Thm:MEC M} Let $X=\{X(p), p\in M\}$ be a centered, unit-variance, isotropic Gaussian random
field satisfying \eqref{Eq:isotropy}, $({\bf H}1)$ and $({\bf H}2)$, where $(M, g)$ is a smooth Riemannian manifold. Then, for a compact, regular stratified, locally convex submanifold $D\subset M$,
\begin{equation}\label{Eq:MEC D}
\E \{\chi(A_u(X,D))\} = \sum_{j=0}^{{\rm dim}(D)}  \left(-2\rho'(0)\right)^{j/2} \mathcal{L}_j (D) \beta_j(u),
\end{equation}
where $\mathcal{L}_j (D)$ are the Lipschitz-Killing curvatures of $D$ and
\begin{equation*}
\begin{split}
\beta_j(u) = \left\{
  \begin{array}{l l}
     (2\pi)^{-1/2}\int_u^\infty e^{-x^2/2}dx & \quad \text{if $j=0$,}\\
     (2\pi)^{-(j+1)/2} H_{j-1} (u) e^{-u^2/2} & \quad \text{if $j=1, \ldots, {\rm dim}(D)$}.
   \end{array} \right.
\end{split}
\end{equation*}
\end{theorem}

\begin{proof}\ By Theorem 12.4.2 in \citep{AT07},
\begin{equation*}
\begin{split}
\E \{\chi(A_u(X,D))\} = \sum_{j=0}^{{\rm dim}(D)} \mathcal{L}_j^X (D) \beta_j(u),
\end{split}
\end{equation*}
where $\mathcal{L}_j^X (D)$ are Lipschitz-Killing curvatures of $D$ under the Riemannian metric $g^X$ induced by $X$. Specifically (see  page 305 in \citep{AT07}),
\begin{equation*}
g_{p_0}^X(\xi_{p_0}, \sigma_{p_0}) := \E\{(\xi_{p_0} X)\cdot(\sigma_{p_0} X)\}=\xi_{p_0}\sigma_{p_0} C(p,q)|_{p=q=p_0}, \quad \forall p_0 \in M,
\end{equation*}
where $\xi_{p_0}, \sigma_{p_0} \in T_{p_0} M$, the tangent spaces of $M$ at $p_0$. Therefore, to prove \eqref{Eq:MEC D}, we only need to show $\mathcal{L}_j^X (D)=\left(-2\rho'(0)\right)^{j/2} \mathcal{L}_j (D)$ for $j=0, \ldots, {\rm dim}(D)$, where $\mathcal{L}_j (D)$ are original Lipschitz-Killing curvatures of $D$ under the Riemannian metric $g$.

We may choose two geodesic curves on $M$, say $\gamma: [0,\delta]\rightarrow M$ and $\tau: [0,\delta]\rightarrow M$, such that $\gamma(t)=\exp_{p_0}(t\xi_{p_0})$ and $\tau(s)=\exp_{p_0}(s\sigma_{p_0})$, where $\exp$ is the usual exponential mapping on $M$. Here $\delta$ is small enough so that the squared distance function $d_M^2(\gamma(t),\tau(s))$, $t, s\in [0,\delta],$ is smooth.

Let $t\in (0,\delta)$ be fixed. For each $s\in [0,\delta]$, consider the minimizing geodesic $\eta_s: [0, 1] \rightarrow M$ joining $\gamma(t)$ to $\tau(s)$, that is, $\eta_s(0) = \gamma(t)$ and $\eta_s(1)=\tau(s)$. Define the variation $f(s, x)=\eta_s(x)$, where $x\in [0,1]$ and $s\in [0, \delta]$, and define the energy function
\[
E(s) :=\int_0^1 \left\langle \frac{\partial f}{\partial x}(s,x), \frac{\partial f}{\partial x}(s,x) \right\rangle dx = d_M^2(\gamma(t), \tau(s)),
\]
where the last equality is due to the fact that $\eta_s(\cdot)$ is geodesic for each fixed $s$. By the first variation formula for energy (see Proposition 2.4 in Chapter 9 of \citep{do Carmo92}), together with again the fact that $\eta_s(\cdot)$ is geodesic, we obtain
\begin{equation*}
\frac{\partial}{\partial s}d_M^2(\gamma(t),\tau(s)) = E'(s) = 2\l \tau'(s), \exp_{\gamma(t)}^{-1}(\tau(s))\r = 2\l \sigma_{p_0}, \exp_{\gamma(t)}^{-1}(\tau(s))\r,
\end{equation*}
where $\exp_{\gamma(t)}^{-1}(\tau(s))$ denotes the vector $\zeta\in T_{\gamma(t)} M$ such that $\tau(s)=\exp_{\gamma(t)}(\zeta)$. In particular, when $s=0$, $\exp_{\gamma(t)}^{-1}(\tau(s))=-t\xi_{p_0}$. Therefore,
\begin{equation}\label{Eq:energy}
\frac{\partial}{\partial s}d_M^2(\gamma(t), \tau(s))|_{s=0} = -2t\l \sigma_{p_0}, \xi_{p_0} \r.
\end{equation}

It follows from \eqref{Eq:energy} that
\begin{equation*}
\begin{split}
\xi_{p_0}\sigma_{p_0} C(p,q)|_{p=q=p_0} &= \frac{\partial}{\partial t}\frac{\partial}{\partial s} \rho\left(d_M^2(\gamma(t), \tau(s))\right)|_{t=s=0} \\
&= -2\frac{\partial}{\partial t} \left[t\l \sigma_{p_0}, \xi_{p_0} \r\rho'\left(d_M^2(\gamma(t), p_0)\right)\right]|_{t=0}\\
& = -2\rho'(0)\l \sigma_{p_0}, \xi_{p_0} \r.
\end{split}
\end{equation*}
Hence the Riemannian metric $g^X$ induced by $X$ is given by
\begin{equation*}
g_{p_0}^X(\xi_{p_0}, \sigma_{p_0}) = -2\rho'(0) \l \xi_{p_0}, \sigma_{p_0} \r, \quad \forall p_0 \in M, \ \xi_{p_0}, \sigma_{p_0} \in T_{p_0} M.
\end{equation*}
By the definition of Lipschitz-Killing curvatures, one has $\mathcal{L}_j^X (D)=\left(-2\rho'(0)\right)^{j/2} \mathcal{L}_j (D)$,
yielding the desired result.
\end{proof}
\begin{remark} The following are some remarks.
\begin{itemize}
\item Let $M=\R^N$, that is, $\{X(p), p\in \R^N\}$ is an isotropic Gaussian random field on Euclidean space. Then, since ${\rm Var}(X_i(p)) = -2\rho'(0)$, we see that (12.2.19) in \citep{AT07} gives the same result as Theorem \ref{Thm:MEC M}.

\item Let $M$ be the $N$-dimensional unit sphere $\S^N$.  For simplicity (the general case of covariance \citep{Cheng:2014,Schoenberg42} can be handled similarly), we assume the covariance of the isotropic Gaussian field on sphere has the form
    \[
    C(p,q)=\sum_{n=0}^\infty b_n \l p, q\r_{\R^{N+1}}^n, \quad p,q\in \S^N,
    \]
    where $b_n\ge 0$ and $\l \cdot, \cdot\r_{\R^{N+1}}$ denotes the usual inner product in $\R^{N+1}$. Notice that $d_{\S^N}(p,q)$ represents the spherical distance (or the greatest circle distance), therefore,
    \[
    C(p,q)=\sum_{n=0}^\infty b_n \left(\cos d_{\S^N}(x,y)\right)^n =\rho\left(d_{\S^N}^2(x,y)\right),
    \]
    where $\rho(r)=\sum_{n=0}^\infty b_n \left(\cos \sqrt{r}\right)^n$. Then $-2\rho'(0)=\sum_{n=1}^\infty nb_n$, which is the same as the parameter $C'$ in \citep{Cheng:2014}. This indicates that the result for isotropic Gaussian fields on sphere in \citep{Cheng:2014} (especially Lemma 3.5 therein) is a special case of Theorem \ref{Thm:MEC M}.

    \item It can be seen from the proof of Theorem \ref{Thm:MEC M} that, under an orthonormal basis of the Riemannian manifold $(M,g)$, say $\{\frac{\partial}{\partial x_i} |_p \}_{1\le i, j\le N}$, one has
        \[
        \E\{X_i(p)X_j(p)\} = -2\rho'(0) \delta_{ij}, \quad \forall p\in M, \ 1\le i, j\le N,
        \]
        where $\delta_{ij}$ is the Kronecker delta function.
\end{itemize}
\end{remark}

Applying Theorem \ref{Thm:MEC M} and Theorem 14.3.3 in \citet{AT07}, we obtain immediately the following approximation for the excursion probability.
\begin{theorem}\label{Thm:Euler approximation}
Let the conditions in Theorem \ref{Thm:MEC M} hold. Then, following the notation therein, there exists $\alpha_0>0$ such that as $u\to \infty$,
\begin{equation*}
\P\left\{\sup_{p\in D} X(p) \ge u \right\} = \sum_{j=0}^{{\rm dim}(D)}  \left(-2\rho'(0)\right)^{j/2} \mathcal{L}_j (D) \beta_j(u)+ o(e^{-\alpha_0 u^2 - u^2/2}).
\end{equation*}
\end{theorem}

The approximation provided in Theorem \ref{Thm:Euler approximation} is very accurate since the error is exponentially smaller than the expected Euler characteristic. In contrast, the approximation for non-smooth Gaussian fields only gives the leading term (i.e., the largest order of polynomials of $u$) of the expected Euler characteristic, which can be seen in \citep{Piterbarg96} and the results in the section below. Therefore, in many statistical applications, especially when using excursion probability for computing p-values, the expected Euler characteristic approximation is much more popular.

\section{Locally Isotropic Gaussian Random Fields on Manifolds}
As before, let $(M,g)$ be a smooth Riemannian manifold of dimension $N$. In this section, we consider a centered Gaussian random field $X=\{X(p), p\in M\}$ such that the covariance function satisfies
\begin{equation}\label{Eq:Pickands condition M}
C(p,q)= 1-c d_M^\alpha (p,q)(1+o(1)) \quad {\rm as }\ d_M (p,q)\to 0,
\end{equation}
where $c>0$ and $\alpha \in (0,2]$. If $X$ is smooth, such as satisfying condition ({\bf H}1), then $\alpha=2$. For $\alpha\in (0,2)$, the field $X$ can be checked to be non-smooth. Covariances satisfying \eqref{Eq:Pickands condition M} behave isotropically in a local sense, therefore we call them (or the Gaussian fields) \emph{locally isotropic}. In particular, if the covariance is isotropic with the representation \eqref{Eq:isotropy}, then \eqref{Eq:Pickands condition M} usually holds. For the Euclidean case, i.e., $M=\R^N$, there is rich literature on studying the asymptotics of excursion probability of Gaussian fields satisfying \eqref{Eq:Pickands condition M} (cf. \citep{Pickands69,QuallsW73,Piterbarg96,ChanL06}). Recently, \citet{Cheng:2014} studied Gaussian fields on sphere satisfying \eqref{Eq:Pickands condition M} by applying the spherical coordinates. Here, we shall use the Riemannian metric itself to show that the locally isotropic property can be generalized as \eqref{Eq:Pickands condition M} for studying the asymptotics of excursion probability of Gaussian fields on Riemannian manifolds.

The main approach is using the argument by (coordinate) charts to transform condition \eqref{Eq:Pickands condition M} to an equivalent condition in Euclidean space (see Lemma \ref{Lem:manidold-Euclidean} below). We first recall some elementary terms for manifolds, which can be found in many textbooks such as \citep{AT07,do Carmo92}. A \emph{chart} on $M$ is a pair $(U, \varphi)$, where $U\subset M$ is open and $\varphi: U \rightarrow \varphi(U)\subset \R^N$ is a homeomorphism. If a collection of charts, say $(U_i, \varphi_i)_{i\in I}$, gives a covering of $M$, i.e., $\cup_{i\in I} U_i= M$, then it is call an \emph{atlas} for $M$. The component functions $x_1, \ldots, x_N$ of a chart $(U, \varphi)$, defined by $\varphi(p) = (x_1(p), \ldots, x_N(p))$, are called the \emph{local coordinates} on $U$.

For two functions $f$ and $h$, we say $f(t) \sim h(t)$ as $t\to t_0\in [-\infty, +\infty]$ if $\lim_{t\to t_0} f(t)/g(t) =1$. Denote by $\|\cdot\|$ the usual Euclidean norm.

\begin{lemma}\label{Lem:manidold-Euclidean}
Let $p, q\in U$, where $p$ is fixed and $(U, \varphi)$ is a chart of $M$. Then as $d_M(p,q) \to 0$,
\[
d_M^2(p,q) \sim \sum_{i, j=1}^N g_{ij}(p)[x_i(q)-x_i(p)][x_j(q)-x_j(p)] = \|G^{1/2}(p)\cdot(\varphi(q)-\varphi(p))\|^2,
\]
where $G^{1/2}(p)$ is the square root of the positive definite matrix $G(p)=(g_{ij}(p))_{1\le i, j\le N}$ defined in Section \ref{Sec:manifold}.
\end{lemma}
\begin{proof}\
Let $\gamma: [0,1] \rightarrow M$ be a smooth curve on $M$ such that $\gamma(0)=p$ and $\gamma([0,1])\subset U$. For $\delta> 0$, denote by $L(\gamma([0,\delta]))$ the length of the segment between $\gamma(0)$ and $\gamma(\delta)$. Since
\[
\gamma'(t) = \sum_{i=1}^N \frac{dx_i(\gamma(t))}{dt}\frac{\partial}{\partial x_i},
\]
we obtain that, as $\delta\to 0$,
\begin{equation*}
\begin{split}
L(\gamma([0,\delta])) &= \int_0^\delta \sqrt{g_{\gamma(t)}(\gamma'(t),\gamma'(t))}dt \\
&= \int_0^\delta \left(\sum_{i,j=1}^Ng_{ij}(\gamma(t))\frac{dx_i(\gamma(t))}{dt}\frac{dx_j(\gamma(t))}{dt}\right)^{1/2}dt\\
&\sim \left(\sum_{i,j=1}^Ng_{ij}(p)[x_i(\gamma(\delta))-x_i(p)][x_j(\gamma(\delta))-x_j(p)]\right)^{1/2}.
\end{split}
\end{equation*}
This implies
\[
d_M(p,q) \sim \left(\sum_{i, j=1}^N g_{ij}(p)[x_i(q)-x_i(p)][x_j(q)-x_j(p)]\right)^{1/2} = \|G^{1/2}(p)\cdot(\varphi(q)-\varphi(p))\|.
\]
\end{proof}

\begin{remark}
Let $M$ be the $N$-dimensional unit sphere $\S^N$. Following the notation of spherical coordinates in \citep{Cheng:2014}, we have
\[
G(p) = {\rm diag}\left(1, (\sin\theta_1)^2, \ldots, \prod_{i=1}^{N-1}(\sin \theta_i)^2\right).
\]
Therefore, Lemma 2.1 in \citep{Cheng:2014} is a special case of Lemma \ref{Lem:manidold-Euclidean}.
\end{remark}

Denote by $H_{\alpha,N}$ the usual Pickands' constant in $\R^N$, that is
\begin{equation*}
H_{\alpha,N}=\lim_{K\to \infty}K^{-N}\int_0^\infty e^u\P\left\{\sup_{s\in [0,K]^N} Z(s)\ge u \right\}du,
\end{equation*}
where $\{Z(s): s\in[0,\infty)^N\}$ is a Gaussian random field such that
\begin{equation*}
\E \big(Z(s) \big)=-\|s\|^\alpha,\ \ \ \ \
{\rm Cov} \big(Z(s),Z(v) \big)=\|s\|^\alpha + \|v\|^\alpha -\|s-v\|^\alpha.
\end{equation*}

Let $T\subset \R^N$ be a bounded $N$-dimensional Jordan measurable set (i.e., the boundary of $T$ has $N$-dimensional Lebesgue measure 0). Let $Y = \{Y(t), t \in \mathbb R^N\}$ be a centered Gaussian field with covariance function $C_Y$ such that for every fixed $t\in T$,
\begin{equation}\label{Eq:locally isotropic property}
C_Y(t,s)= 1-\|B(t)(s-t)\|^\alpha (1+o(1)) \quad {\rm as}\ \|s-t\|\to 0,
\end{equation}
where $B(t)$ is a non-degenerate $N\times N$ matrix and $\alpha \in (0,2]$.

We will make use of the following result, which is a direct consequence of Theorem 2.1 in \citep{ChanL06} (see also Theorem 2.2 in \citep{Cheng:2014}) and Lemma 2.3 in \citep{Cheng:2014}.

\begin{theorem}\label{Thm:Pickands approximation}
Let $T\subset \R^N$ be a bounded $N$-dimensional Jordan measurable set. Suppose the centered Gaussian random
field $\{Y(t), t\in \R^N\}$ satisfies condition (\ref{Eq:locally isotropic property}). Then as $u\to \infty$,
\begin{equation*}
\P\left\{\sup_{t\in T} Y(t) \ge u \right\} = \left(\int_T|{\rm det}(B(t))|dt\right) H_{\alpha,N} u^{2N/\alpha} \Psi(u) (1+o(1)).
\end{equation*}
Here and in the sequel, $ \Psi(u) = (2\pi)^{-1/2}\int_u^\infty e^{-v^2/2}dv$.
\end{theorem}

Now we are ready to prove our main result of this section.
\begin{theorem}\label{Thm:Pickands approximation M}
Let $X=\{X(p), p\in M\}$ be a centered Gaussian random field satisfying condition (\ref{Eq:Pickands condition M}), where $(M, g)$ is an $N$-dimensional smooth Riemannian manifold. Let $D\subset M$ be an $N$-dimensional smooth compact submanifold on $M$. Then as $u\to \infty$,
\begin{equation*}
\P\left\{\sup_{p\in D} X(p) \ge u \right\} = {\rm Vol}(D)c^{N/\alpha}H_{\alpha,N} u^{2N/\alpha} \Psi(u) (1+o(1)),
\end{equation*}
where ${\rm Vol}(D)$ denotes the volume of $D$ and $c>0$ is the constant in (\ref{Eq:Pickands condition M}).
\end{theorem}

\begin{proof}\
Let $(U_i, \varphi_i)_{i\in I}$ be an atlas of $M$. Since $D$ is compact, it has a finite covering $(U_i, \varphi_i)_{i\in I_0}$, where $I_0\subset I$ is a finite index set. For $i\in I_0$, define $\bar{X} : \varphi_i(U_i)\subset \R^N \rightarrow \R^N$ by $\bar{X}=X\circ \varphi_i^{-1}$. By Lemma \ref{Lem:manidold-Euclidean}, the covariance of $\bar{X}$, denoted by $\bar{C}(\cdot, \cdot): \varphi_i(U_i)\times \varphi_i(U_i) \subset \R^N\times \R^N \rightarrow \R$, satisfies the following property: for every fixed $\varphi_i(p)\in \varphi_i(U_i)$, as $\|\varphi_i(q)-\varphi_i(p)\|\to 0$,
\begin{equation*}
\begin{split}
\bar{C}(\varphi_i(p),\varphi_i(q)) &= 1-c \|(G^{1/2}\circ \varphi_i^{-1})(\varphi_i(p))\cdot(\varphi_i(q)-\varphi_i(p))\|^\alpha (1+o(1))\\
&= 1- \|c^{1/\alpha}(G^{1/2}\circ \varphi_i^{-1})(\varphi_i(p))\cdot(\varphi_i(q)-\varphi_i(p))\|^\alpha (1+o(1)).
\end{split}
\end{equation*}
We make a decomposition $D=\cup_{j=1}^m D_j$ such that any neighboring $D_j$ and $D_{j'}$ belong to a same chart. Suppose $D_j \subset U_i$, applying Theorem \ref{Thm:Pickands approximation} yields
\begin{equation}\label{Eq:app-UD}
\begin{split}
\P&\left\{\sup_{p\in D_j} X(p) \ge u \right\} = \P\left\{\sup_{x\in \varphi_i(D_j)} \bar{X}(x) \ge u \right\}\\
&=\left(\int_{\varphi_i(D_j)}\Big|{\rm det}\left(c^{1/\alpha}\left(G^{1/2}\circ \varphi_i^{-1}\right)(x)\right)\Big|dx\right) H_{\alpha,N} u^{2N/\alpha} \Psi(u) (1+o(1))\\
&=\left(\int_{\varphi_i(D_j)}\big|{\rm det}\left(\left(G\circ \varphi_i^{-1}\right)(x)\right)\big|^{1/2}dx\right) c^{N/\alpha}H_{\alpha,N} u^{2N/\alpha} \Psi(u) (1+o(1))\\
&={\rm Vol}(D_j) c^{N/\alpha}H_{\alpha,N} u^{2N/\alpha} \Psi(u) (1+o(1)),
\end{split}
\end{equation}
where the last line is due to the definition of volume on $(M, g)$.

It is obvious that
\begin{equation}\label{Eq:upper-bound}
\begin{split}
\P\left\{\sup_{p\in D} X(p) \ge u \right\} \le \sum_{j=1}^m \P\left\{\sup_{p\in D_j} X(p) \ge u \right\}.
\end{split}
\end{equation}
On the other hand, by the Bonferroni inequality,
\begin{equation}\label{Eq:lower-bound}
\begin{split}
\P&\left\{\sup_{p\in D} X(p) \ge u \right\} \ge \sum_{j=1}^m \P\left\{\sup_{p\in D_j} X(p) \ge u \right\}  - \sum_{j\ne j'} \P\left\{\sup_{p\in D_j} X(p) \ge u, \sup_{q\in D_{j'}} X(q) \ge u \right\}.
\end{split}
\end{equation}
If $D_j$ and $D_{j'}$ are not adjacent, i.e., $d_M(D_j, D_{j'})>0$, then it follows from the Borell-TIS inequality (cf. Theorem 2.1.1 in \citep{AT07}) that
\begin{equation*}
\begin{split}
&\lim_{u\to \infty}\frac{\P\left\{\sup_{p\in D_j} X(p) \ge u, \sup_{q\in D_{j'}} X(q) \ge u \right\}}{u^{2N/\alpha}\Psi(u)}\\
&\quad \le \lim_{u\to \infty}\frac{\P\left\{\sup_{p\in D_j, q\in D_{j'} } [X(p) + X(q)]/2 \ge u \right\}}{u^{2N/\alpha}\Psi(u)} = 0.
\end{split}
\end{equation*}
If $D_j$ and $D_{j'}$ are adjacent, then by our assumption, there exists a chart $(U_i, \varphi_i)$ containing both $D_j$ and $D_{j'}$. Since
\begin{equation*}
\begin{split}
\P&\left\{\sup_{p\in D_j} X(p) \ge u, \sup_{q\in D_{j'}} X(q) \ge u \right\}\\
&= \P\left\{\sup_{p\in D_j} X(p) \ge u\right\} + \P\left\{\sup_{q\in D_{j'}} X(q) \ge u \right\} - \P\left\{\sup_{p\in D_j \cup D_{j'}} X(p) \ge u\right\},
\end{split}
\end{equation*}
applying \eqref{Eq:app-UD} to the last three terms respectively yields that the joint excursion probability above is $o(u^{2N/\alpha}\Psi(u))$. Therefore, the last sum in \eqref{Eq:lower-bound} is $o(u^{2N/\alpha}\Psi(u))$.

Now, combining \eqref{Eq:app-UD} with \eqref{Eq:upper-bound} and \eqref{Eq:lower-bound}, we obtain
\begin{equation*}
\begin{split}
\P\left\{\sup_{p\in D} X(p) \ge u \right\} &= \sum_{j=1}^m {\rm Vol}(D_j) c^{N/\alpha}H_{\alpha,N} u^{2N/\alpha} \Psi(u) (1+o(1))\\
&={\rm Vol}(D) c^{N/\alpha}H_{\alpha,N} u^{2N/\alpha} \Psi(u) (1+o(1)).
\end{split}
\end{equation*}
\end{proof}

In Theorem \ref{Thm:Pickands approximation M}, the submanifold $D\subset M$ is assumed to be of the same dimension as $M$. This is because under such assumption we can apply the Riemannian metric $g$ on $M$ directly. If ${\rm dim}(D)<N$, we may turn to the restriction of $g$ on the tangent space of $D$, which can be regarded as the Riemannian metric on $D$ induced from $(M,g)$. We denote such \emph{induced Riemannian metric} on $D$ by $\tilde{g}$. Then we have following result as an extension of Theorem \ref{Thm:Pickands approximation M}.

\begin{corollary}\label{Cor:Pickands approximation Dk}
Let $X=\{X(p), p\in M\}$ be a centered Gaussian random field satisfying condition (\ref{Eq:Pickands condition M}), where $(M, g)$ is an $N$-dimensional smooth Riemannian manifold. Let $D\subset M$ be a $k$-dimensional smooth compact submanifold on $M$, where $0<k\le N$. Then as $u\to \infty$,
\begin{equation*}
\P\left\{\sup_{p\in D} X(p) \ge u \right\} = {\rm Vol}(D)c^{k/\alpha}H_{\alpha,k} u^{2k/\alpha} \Psi(u) (1+o(1)),
\end{equation*}
where ${\rm Vol}(D)$ denotes the volume of $D$ under the induced Riemannian metric $\tilde{g}$ and $c>0$ is the constant in (\ref{Eq:Pickands condition M}).
\end{corollary}
\begin{proof}\
Let $(V_i, \phi_i)_{i\in I_0}$ be a finite atlas of $D$. For $i\in I_0$, defnie $\bar{X} : \phi_i(V_i)\subset \R^k \rightarrow \R^k$ by $\bar{X}=X\circ \phi_i^{-1}$. Similarly to Lemma \ref{Lem:manidold-Euclidean}, we have that the covariance of $\bar{X}$, denoted by $\bar{C}(\cdot, \cdot): \phi_i(V_i)\times \phi_i(V_i) \subset \R^k\times \R^k \rightarrow \R$, satisfies the following property: for every fixed $\phi_i(p)\in \phi_i(V_i)$, as $\|\phi_i(q)-\phi_i(p)\|\to 0$,
\begin{equation*}
\begin{split}
\bar{C}(\phi_i(p),\phi_i(q)) &= 1-c \|(\tilde{G}^{1/2}\circ \phi_i^{-1})(\phi_i(p))\cdot(\phi_i(q)-\phi_i(p))\|^\alpha (1+o(1))\\
&= 1- \|c^{1/\alpha}(\tilde{G}^{1/2}\circ \phi_i^{-1})(\phi_i(p))\cdot(\phi_i(q)-\phi_i(p))\|^\alpha (1+o(1)),
\end{split}
\end{equation*}
where $\tilde{G}(p)=(\tilde{g}_{ij}(p))_{1\le i, j\le k}$. The desired result then follows from similar arguments in the proof of Theorem \ref{Thm:Pickands approximation M}.
\end{proof}

\begin{remark}
If $X=\{X(p), p\in M\}$ is isotropic such that conditions ({\bf H}1) and ({\bf H}2) are fulfilled, then the covariance $C(p,q) = \rho\left(d_M^2(p,q)\right)$ satisfies condition \eqref{Eq:Pickands condition M} with $\alpha=2$ and $c=-\rho'(0)$. Since Pickands' constant $H_{2,N}=\pi^{-N/2}$, one can check that the approximation in Theorem \ref{Thm:Pickands approximation M} only provides the leading term of the approximation in Theorem \ref{Thm:Euler approximation}. This is because, Theorem \ref{Thm:Pickands approximation M} does not take into account the boundary effect of $D$. Therefore, the expected Euler characteristic approximation in Theorem \ref{Thm:Euler approximation} is much more accurate for smooth isotropic Gaussian fields. However, Theorem \ref{Thm:Pickands approximation M} is still very useful since it provides an approximation to non-smooth locally isotropic Gaussian fields, for which Theorem \ref{Thm:Euler approximation} is not applicable due to lack of smoothness.
\end{remark}

In this paper, for the non-smooth case, we only study covariances with a simple local structure \eqref{Eq:Pickands condition M}. However, similarly to the Euclidean case \citep{Piterbarg96}, one can also investigate the asymptotics for Gaussian covariances with more general local structures, such as letting the constants $c$ and $\alpha$ in \eqref{Eq:Pickands condition M} be functions $c(p)$ and $\alpha(p)$ depending on $p$ \citep{Istas13}. We leave these problems for future research.

\bibliographystyle{plain}
\begin{small}

\end{small}

\bigskip

\begin{quote}
\begin{small}
\textsc{Dan Cheng}\\
Department of Statistics, North Carolina State
University\\
2311 Stinson Drive, Campus Box 8203, Raleigh, NC 27695, U.S.A.\\
E-mail: \texttt{dcheng2@ncsu.edu}
\end{small}
\end{quote}

\end{document}